\documentclass{article}
\usepackage{amssymb}
\usepackage{amsmath}
\usepackage{amsthm}
\usepackage{color}
\usepackage[dvipsnames]{xcolor}

\newtheorem{theorem}{Theorem}[section]
\newtheorem{lemma}[theorem]{Lemma}
\newtheorem{proposition}[theorem]{Proposition}
\newtheorem{corollary}[theorem]{Corollary}

\theoremstyle{definition}

\newtheorem{example}[theorem]{Example}
\newtheorem{algorithm}[theorem]{Algorithm}

\theoremstyle{remark}
\newtheorem{remark}[theorem]{Remark}

\numberwithin{equation}{section}

\begin{document}

\title{On a transport problem and monoids of non-negative integers}

\author{Aureliano M. Robles-P\'erez\thanks{Both authors are supported by the project MTM2014-55367-P, which is funded by Mi\-nis\-terio de Econom\'{\i}a y Competitividad and Fondo Europeo de Desarrollo Regional FEDER, and by the Junta de Andaluc\'{\i}a Grant Number FQM-343. The second author is also partially supported by Junta de Andaluc\'{\i}a/Feder Grant Number FQM-5849.} \thanks{Departamento de Matem\'atica Aplicada, Universidad de Granada, 18071-Granada, Spain. \newline E-mail: {\bf arobles@ugr.es}}
\mbox{ and} Jos\'e Carlos Rosales$^*$\thanks{Departamento de \'Algebra, Universidad de Granada, 18071-Granada, Spain. \newline E-mail: {\bf jrosales@ugr.es}} }

\date{ }

\maketitle

\begin{abstract}
	A problem about how to transport profitably a group of cars leads us to study the set $T$ formed by the integers $n$ such that the system of inequalities, with non-negative integer coefficients,
	$$a_1x_1 +\cdots+ a_px_p + \alpha \leq n \leq b_1x_1 +\cdots+ b_px_p - \beta$$
	has at least one solution in ${\mathbb N}^p$. We will see that $T\cup\{0\}$ is a submonoid of $({\mathbb N},+)$. Moreover, we show algorithmic processes to compute $T$.	
\end{abstract}
\noindent {\bf Keywords:} Transport problem, Diophantine inequalities, monoids.

\medskip

\noindent{\it 2010 AMS Classification:} 20M14, 11D75

\section{Introduction}\label{sect1}

A transport company is dedicated to carry cars from the factory to the authorised dealer. For that, the company uses small and large trucks with a capacity of three and six cars, respectively. Moreover, we have the following determinants.
\begin{itemize}
	\item The use of those trucks represents, for the company, a cost of 1200 and 1500 euros, respectively.
	\item The company charges to the dealer 300 euros for each transported car.
	\item Without charge for the customer, an additional car is loaded facing possible eventualities.
\end{itemize}
If the company consider that a transport order is cost-effective when it has profits of at least 900 euros, how many cars must be transported in order to achieve that purpose?

It is clear that a transport order is profitable if and only if there exist $x,y \in {\mathbb N}$ (where ${\mathbb N}$ is the set of non-negative integers) such that
\begin{equation}\label{ineqsyst0}
\left. \begin{array}{c}
300n \geq 1200x + 150y + 900 \\[1mm] n+1 \leq 3x + 6y
\end{array} \right\}.			
\end{equation}
Simplifying the first inequality of \eqref{ineqsyst0}, we have the equivalent system
\begin{equation}\label{ineqsyst1}
\left. \begin{array}{c}
n \geq 4x + 5y + 3 \\[1mm] n \leq 3x + 6y -1	
\end{array} \right\}.			
\end{equation}
Consequently, the set $\left\{ n \in {\mathbb N} \mid \mbox{\eqref{ineqsyst1} has a solution in } {\mathbb N}^2 \right\}$ is formed by the non-negative integers which give us an affirmative answer to the proposed problem.

We can generalize the above problem in the following way: if we consider $a=(a_1,\ldots,a_p), b=(b_1,\ldots,b_p) \in {\mathbb N}^p$ and $\alpha,\beta \in {\mathbb N}$, then we want to compute the set
\begin{equation}\label{ineqsyst2}
S(a,b,\alpha,\beta)=\left\{ \begin{array}{c|c} 
\!\!\! n\in {\mathbb N} \! & \begin{array}{c}
\!\!\! a_1x_1 +\cdots+ a_px_p + \alpha \leq n \leq b_1x_1 +\cdots+ b_px_p - \beta \\[3pt] \mbox{ for some } (x_1,\ldots x_p) \in {\mathbb N}^p
\end{array}
\end{array} \!\!\!\!\! \right\}.
\end{equation}

In this work we will prove that $S(a,b,\alpha,\beta)\cup\{0\}$ is a submonoid of $({\mathbb N},+)$ (that is, a subset of ${\mathbb N}$ that is closed under addition and contains the zero element) and our main purpose will be to give an algorithm in order to compute the minimal system of generators for such a monoid.

To finish this introduction, let us observe that, with the above notation, the problem studied in \cite{forum} is associated to the particular sets $S(a,b,1,1)$.

\section{First results and the case $(\alpha,\beta)=(0,0)$}\label{sect2}

In this section will be suppose that $a=(a_1,\ldots,a_p)$, $b=(b_1,\ldots,b_p)$ belong to ${\mathbb N}^p$ and that $\alpha,\beta \in {\mathbb N}$.

\begin{lemma}\label{lem1}
	If $m,n \in S(a,b,\alpha,\beta)$, then $m+n \in S(a,b,\alpha,\beta)$.
\end{lemma}

\begin{proof}
	If $m,n \in S(a,b,\alpha,\beta)$, then there exist $(x_1,\ldots,x_p), (y_1,\ldots,y_p) \in {\mathbb N}^p$ such that 
	$$a_1x_1 +\cdots+ a_px_p + \alpha \leq m \leq b_1x_1 +\cdots+ b_px_p - \beta$$
	and
	$$a_1y_1 +\cdots+ a_py_p + \alpha \leq n \leq b_1y_1 +\cdots+ b_py_p - \beta.$$
	Therefore,
	$$a_1(x_1+y_1) +\cdots+ a_p(x_p+y_p) + 2\alpha \leq m+n \leq b_1(x_1+y_1) +\cdots+ b_p(x_p+y_p) - 2\beta.$$
	Thus,
	$$a_1(x_1+y_1) +\cdots+ a_p(x_p+y_p) + \alpha \leq m+n \leq b_1(x_1+y_1) +\cdots+ b_p(x_p+y_p) - \beta,$$
	and, consequently, $m+n \in S(a,b,\alpha,\beta)$.
\end{proof}

As an immediate consequence of the above result we have the following proposition.

\begin{proposition}\label{prop2}
	$S(a,b,\alpha,\beta) \cup \{0\}$ is a submonoid of $({\mathbb N},+)$.
\end{proposition}

Let $X$ be a non-empty subset of ${\mathbb N}^k$. We will denote by $\langle X \rangle$ the submonoid of $({\mathbb N}^k,+)$ generated by $X$, that is, 
$$\langle X \rangle = \left\{ \lambda_1x_1+\cdots+\lambda_nx_n \mid n \in {\mathbb N} \setminus \{0\}, \; x_1,\ldots,x_n \in X, \; \lambda_1,\ldots,\lambda_n \in {\mathbb N} \right\}.$$
If $M$ is a submonoid of $({\mathbb N}^k,+)$ and  $M = \langle X \rangle$, then we will say that $X$ is a \emph{system of generators} of (for) $M$ or, equivalently, that $M$ is generated by $X$. In addition, if no proper subset of $X$ generates $M$, then we will say that $X$ is a \emph{minimal system of generators} of (for) $M$.

The next result is \cite[Corollary~2.8]{springer}.

\begin{proposition}\label{prop3}
	Every submonoid of $({\mathbb N},+)$ admit a unique minimal system of generators. Moreover, such a system is finite.	
\end{proposition}

Our objective in this work is to show an algorithm to compute the minimal system of generators of $S(a,b,\alpha,\beta) \cup \{0\}$. In this section we will solve the problem when $(\alpha,\beta)=(0,0)$.

A submonoid $M$ of $({\mathbb N}^p,+)$ is \emph{finitely generated} if there exists a finite set $X$ such that $M = \langle X \rangle$. From Proposition~\ref{prop3} we know that every  submonoid of $({\mathbb N},+)$ is finitely generated. However, if $k \geq 2$, then there exist submonoids of $({\mathbb N}^k,+)$ which are not finitely generated.

Let $z=(z_1,\ldots,z_p) \in {\mathbb Z}^p$ (where ${\mathbb Z}$ is the set of integers) and let $A(z)=\left\{(x_1,\ldots,x_p) \in {\mathbb N}^p \mid z_1x_1+\cdots+z_px_p \geq 0 \right\}$. It is well known that $A(z)$ is a fini\-te\-ly generated submonoid of $({\mathbb N}^p,+)$ and, moreover, in \cite{ajili} it is shown an algorithm to compute a finite system of generators of $A(z)$.

If $s,t\in {\mathbb Z}$ (with $s\leq t$), then we will denote by $[s,t] = \left\{x\in {\mathbb Z} \mid s \leq x \leq t \right\}$.

\begin{theorem}\label{thm4}
	Let $\{m_1,\ldots,,m_q \}$ be a system of generators of $A(b-a)$, where $m_i=(m_{i1},\ldots,m_{ip})$ for all $i\in \{1,\ldots,q\}$. Then $S(a,b,0,0) \cup \{0\}$ is the submonoid of $({\mathbb N},+)$ generated by	${\mathrm L}=\bigcup_{i=1}^q [a_1m_{i1}+\cdots+a_pm_{ip},b_1m_{i1}+\cdots+b_pm_{ip}]$.
\end{theorem}

\begin{proof}
	It is clear that
		$$S(a,b,0,0) = \bigcup_{(x_1,\ldots,x_p) \in A(b-a)} [a_1x_1+\cdots+a_px_p,b_1x_1+\cdots+b_px_p].$$
	Therefore, $S(a,b,0,0) \cup \{0\}$ is a submonoid of $({\mathbb N},+)$ which contains $L$. In order to finish the proof, we will see that, if $T$ is a submonoid of $({\mathbb N},+)$ which contains ${\mathrm L}$, then $S(a,b,0,0) \subseteq T$. For that we will prove that, if $(x_1,\ldots,x_p) \in A(b-a)$, then $[a_1x_1+\cdots+a_px_p,b_1x_1+\cdots+b_px_p] \subseteq T$.
	
	If $(x_1,\ldots,x_p) \in A(b-a)$, then we have $(x_1,\ldots,x_p)=\lambda_1 m_1 + \cdots + \lambda_q m_q$ for some $\lambda_1, \ldots, \lambda_q \in {\mathbb N}$. By induction over $\lambda_1 + \cdots + \lambda_q$, we will show that $[a_1x_1+\cdots+a_px_p,b_1x_1+\cdots+b_px_p] \subseteq T$. Thus, if $\lambda_1 + \cdots + \lambda_q=0$ the result is obvious. Let us suppose that $\lambda_1 + \cdots + \lambda_q \geq 1$ and let $i\in \{1,\ldots,q \}$ such that $\lambda_i\not=0$. If we take $(y_1,\ldots,y_p) = \lambda_1 m_1 + \cdots + (\lambda_i-1)m_i + \cdots + \lambda_q m_q$, then $(x_1,\ldots,x_p) = (y_1,\ldots,y_p) + m_i$ and, by the induction hypothesis, we get that $[a_1y_1+\cdots+a_py_p,b_1y_1+\cdots+b_py_p] \subseteq T$. Since $T$ is a monoid, then $[a_1y_1+\cdots+a_py_p,b_1y_1+\cdots+b_py_p] + [a_1m_{i1}+\cdots+a_pm_{ip},b_1m_{i1}+\cdots+b_pm_{ip}] \subseteq T$ and, consequently, $[a_1x_1+\cdots+a_px_p,b_1x_1+\cdots+b_px_p] \subseteq T$.
\end{proof}

Now, we are going to describe an algorithmic process, given in \cite{forum}, in order to compute a system of generators for ${\rm A}(z)$. Thereby, we get a self-contained paper and, in addition, we will be able to describe examples without necessity of referencing to \cite{ajili}.

Let ${\rm B}(z)= \{(x_1,\ldots,x_p,x_{p+1})\in {\mathbb N}^{p+1} \mid z_1x_1+\cdots+z_px_p-x_{p+1}=0 \}$. It is well known (see \cite{monoides}) that ${\rm B}(z)$ is a finitely generated submonoid of $({\mathbb N}^{p+1},+)$ and, in addition, its set of minimal generators coincide with the minimal elements (with the usual order in ${\mathbb N}^{p+1}$) of the set ${\rm B}(z) \setminus \{(0,\ldots,0)\}$. Moreover, we know that, if $(x_1,\ldots,x_p,x_{p+1})$ is a minimal element of ${\rm B}(z) \setminus \{(0,\ldots,0)\}$, then $x_1+\cdots+x_{p+1} \leq |z_1|+\cdots+|z_p|+2$ (Pottier's bound, \cite{pottier}). Finally, it is easy to see that, if $\{b_1,\ldots,b_q\}$ is a system of generators of ${\rm B}(z)$, then $\{\pi(b_1),\ldots,\pi(b_q)\}$ is a system of generators for ${\rm A}(z)$ (where $\pi(x_1,\ldots,x_p,x_{p+1})=(x_1,\ldots,x_p)$). Therefore, we have an algorithm to compute a system of generators for ${\rm A}(z)$.

\begin{example}\label{exmp5}
	We are going to compute the minimal system of generators for $S=S\big((4,5),(3,6),0,0\big)\cup\{0\}$. We begin computing a system of generators for $A(-1,1)=\{(x,y) \in {\mathbb N}^2 \mid -x+y \geq 0 \}$. For that, we compute the minimal elements of $B(-1,1) \setminus \{(0,0,0)\}$, where $B(-1,1)=\{ (x,y,z) \in {\mathbb N}^3 \mid -x+y-z=0 \}$. By applying the Pottier's bound, we have that $(1,1,0),(0,1,1)$ are such minimal elements and, in consequence, $\{(1,1),(0,1)\}$ is a system of generators for $A(-1,1)$. Therefore, by applying Theorem~\ref{thm4}, we conclude that $[9,9] \cup [5,6] =\{5,6,9\}$ is a system of generators for $S$. Thus, $S=\{0,5,6,9,10,11,12,14,\to\}$ (where the symbol $\to$ means that every number greater than 14 belongs to $S$).
\end{example}

\section{The case $(\alpha,\beta)\not=(0,0)$}\label{sect3}

Let ${\mathcal A}=\left\{(x_1,\ldots,x_p) \in {\mathbb N}^p \mid (b_1-a_1)x_1 + \cdots + (b_p-a_p)x_p \geq \alpha+\beta \right\}$. Then it is clear that 
$$S(a,b,\alpha,\beta) = \bigcup_{(x_1,\ldots,x_p) \in {\mathcal A}} [a_1x_1+\cdots+a_px_p+\alpha,b_1x_1+\cdots+b_px_p-\beta].$$
Let us define over ${\mathcal A}$ the binary relation $\leq_{A(b-a)}$ as follow:
$$x \leq_{A(b-a)} y \;\mbox{ if }\; y-x \in A(b-a).$$
In \cite{london} it is shown that $\leq_{A(b-a)}$ is an order relation. Moreover, from the results in such a work, we also know that, if ${\mathcal D}$ is a system of generators of $A(b-a)$ and ${\mathcal C}=\mathrm{minimals}_{\leq_{A(b-a)}}{\mathcal A}$, then ${\mathcal A} = {\mathcal C} + \langle {\mathcal D} \rangle$. Furthermore, in \cite{london} it is also given an algorithm to compute ${\mathcal C}$ and ${\mathcal D}$.

Let $B \subseteq {\mathbb N}$. We will say that $M$ is a \emph{$B$-monoid} if $M$ is a submonoid of $({\mathbb N},+)$ fulfilling that $(M\setminus \{0\}) + B \subseteq M$.

It is clear that ${\mathbb N}$ is a $B$-monoid and that the intersection of $B$-monoids is another $B$-monoid. Therefore, we can give the notion of \emph{smallest $B$-monoid which contains a given set $A$ of non-negative integers}.

\begin{theorem}\label{thm6}
	With the above notation, let $${\mathrm C} = \bigcup_{(c_1,\ldots,c_p) \in {\mathcal C}} [a_1c_1+\cdots+a_pc_p+\alpha, b_1c_1+\cdots+b_pc_p-\beta]$$ and let $${\mathrm D} = \bigcup_{(d_1,\ldots,d_p) \in {\mathcal D}} [a_1d_1+\cdots+a_pd_p, b_1d_1+\cdots+b_pd_p].$$ Then $S(a,b,\alpha,\beta) \cup \{0\}$ is the smallest ${\mathrm D}$-monoid which contains ${\mathrm C}$.
\end{theorem}

\begin{proof}
	It is clear that $S(a,b,\alpha,\beta)\cup\{0\}$ is a ${\mathrm D}$-monoid which contains ${\mathrm C}$. Now, let us see that, if $T$ is a ${\mathrm D}$-monoid containing ${\mathrm C}$, then $S(a,b,\alpha,\beta) \cup \{0\} \subseteq T$. For that, we will prove that, if $(x_1,\ldots,x_p) \in {\mathcal A}$, then
	\begin{equation}\label{eqaux6}
	[a_1x_1+\cdots+a_px_p+\alpha,b_1x_1+\cdots+b_px_p-\beta] \subseteq T.
	\end{equation}
	
	Let us suppose that ${\mathcal D}=\{d_1,\ldots,d_q \}$, where $d_i=(d_{i1},\ldots,d_{ip})$ for all $i\in\{1,\ldots,q \}$. Since $(x_1,\ldots,x_p) \in {\mathcal A}$, then there exist $c\in {\mathcal C}$ and $\lambda_1,\ldots,\lambda_q \in {\mathbb N}$ such that $(x_1,\ldots,x_p)=c+\lambda_1 d_1 + \cdots + \lambda_q d_q$. Now, in order to show \eqref{eqaux6}, we use induction over $\lambda_1 + \cdots + \lambda_q$ following the reasoning exposed in Theorem~\ref{thm4}.
\end{proof}

At this moment, we propose to give an algorithm that allows us to compute the smallest $B$-monoid containing a given set $A$.

\begin{proposition}\label{prop7}
	Let $M$ be a submonoid of $({\mathbb N},+)$ and let $A=\{a_1,\ldots,a_n \} \subseteq {\mathbb N} \setminus \{0\}$ be a system of generators of $M$. Then $M$ is a $B$-monoid if and only if $A+B \subseteq M$.
\end{proposition}

\begin{proof}
	The necessary condition is trivial. So, let us see the sufficient one.
	
	If $m \in M \setminus \{0\}$, then there exists $(\lambda_1,\ldots,\lambda_n) \in {\mathbb N}^n \setminus \{(0,\ldots,0)\}$ such that $m=\lambda_1 a_1 + \cdots + \lambda_n a_n$. Now, by induction over $\lambda_1 + \cdots + \lambda_n$, we will prove that $m+b \in M$ for all $b \in B$. Firstly, the result is trivially true for $\lambda_1 + \cdots + \lambda_n=1$. Let us suppose that $\lambda_1 + \cdots + \lambda_n \geq 2$ and let $i\in \{1,\ldots,n \}$ such that $\lambda_i \not=0$. By the induction hypothesis, we easily deduce that $(m-a_i)+b \in M$. Therefore, $(m-a_i)+b+a_i \in M$ and, consequently, $m+b \in M$.
\end{proof}

If $M$ is a submonoid of $({\mathbb N},+)$, we will denote by ${\rm msg}(M)$ the minimal system of generators of $M$.

\begin{algorithm}\label{alg8}
	{\textrm
		INPUT: A finite set $A$ of positive integers. \par
		OUTPUT: The minimal system of generators for the smallest $B$-monoid containing $A$. \par
		(1) $X={\rm msg}(\langle A \rangle)$. \par
		(2) $Y=X \cup \left( X+B \right)$. \par
		(3) If ${\rm msg}(\langle Y \rangle)=X$, then return $X$. \par
		(4) Set $X={\rm msg}(\langle Y \rangle)$ and go to (2).
	}
\end{algorithm}

By using Propositon~\ref{prop7}, it is easy to see that the above algorithm operates suitably. On the other hand, observe that the most complex process in Algorithm~\ref{alg8} is to compute ${\rm msg}(\langle Y \rangle)$, that is, compute the minimal system of generators for a monoid starting from any system of generators of it. For this purpose, we can use the \texttt{GAP} package \texttt{numericalsgps} (see \cite{numericalsgps}).

\begin{example}\label{exmp9}
	By using Algorithm~\ref{alg8}, we are going to compute the smallest $\{2,3\}$-monoid containing $\{5,7\}$.
	\begin{itemize}
		\item $X=\{5,7\}$.
		\item $Y=\{5,7,8,9,10\}$.
		\item ${\rm msg}(\langle Y \rangle)=\{5,7,8,9\}$.
		\item $X=\{5,7,8,9\}$.
		\item $Y=\{5,7,8,9,10,11,12\}$.
		\item ${\rm msg}(\langle Y \rangle)=\{5,7,8,9,11\}$.
		\item $X=\{5,7,8,9,11\}$.
		\item $Y=\{5,7,8,9,10,11,12,13,14\}$.
		\item ${\rm msg}(\langle Y \rangle)=\{5,7,8,9,11\}=X$.
	\end{itemize}
	Therefore, $\langle 5,7,8,9,11 \rangle = \{0,5,7,\to \}$ is the smallest $\{2,3\}$-monoid containing $\{5,7\}$.
\end{example}

As a consequence of the previous results, we have an algorithm which allows us to compute the minimal system of generators of the monoid $S(a,b,\alpha,\beta) \cup \{0\}$.

\begin{algorithm}\label{alg10}
	{\textrm
		INPUT: $a,b \in {\mathbb N}^p$ and $\alpha,\beta \in {\mathbb N}$ such that $(\alpha,\beta)\not=(0,0)$. \par
		OUTPUT: The minimal system of generators of $S(a,b,\alpha,\beta) \cup \{0\}$. \par
		(1) Compute a finite system of generators ${\mathcal D}$ for $A(b-a)$. \par
		(2) Compute ${\mathcal C}=\mathrm{minimals}_{\leq_{A(b-a)}}{\mathcal A}$, where
		$$\mbox{} \hspace{20pt} {\mathcal A}=\left\{(x_1,\ldots,x_p) \in {\mathbb N}^p \mid (b_1-a_1)x_1 + \cdots + (b_p-a_p)x_p \geq \alpha+\beta \right\}.$$ \par
		(3) ${\mathrm C} = \bigcup_{(c_1,\ldots,c_p) \in {\mathcal C}} [a_1c_1+\cdots+a_pc_p+\alpha, b_1c_1+\cdots+b_pc_p-\beta]$ and \par
		\mbox{} \hspace{9.75pt} ${\mathrm D} = \bigcup_{(d_1,\ldots,d_p) \in {\mathcal D}} [a_1d_1+\cdots+a_pd_p, b_1d_1+\cdots+b_pd_p]$. \par
		(4) Return the minimal system of generators of the smallest ${\mathrm D}$-monoid containing ${\mathrm C}$.
	}
\end{algorithm}

In Section~\ref{sect2}, we have described an algorithmic process to compute ${\mathcal D}$. At the beginning of this section, we also have commented that, from the results of \cite{london}, we have an algorithm to compute ${\mathcal C}$. In order to give a self-contained paper and to show examples without necessity of referencing to \cite{london}, we are going to mention briefly how to build ${\mathcal C}$. For that, we need to consider the following equations and inequalities:
\begin{align}
(b_1-a_1)x_1+\cdots+(b_p-a_p)x_p \geq \alpha + \beta, \label{eq:1}\\
(b_1-a_1)x_1+\cdots+(b_p-a_p)x_p-x_{p+1} - (\alpha + \beta)x_{p+2} = 0, \label{eq:2}\\
(b_1-a_1)x_1+\cdots+(b_p-a_p)x_p \geq 0. \label{eq:3}
\end{align} 

The next result is a direct consequence of the results from \cite[Section~2]{london}.

\begin{proposition}\label{prop11}
	Let $A = \{\alpha_1,\ldots,\alpha_t \}$, with $\alpha_i = (\alpha_{i_1},\ldots,\alpha_{i_{p+2}})$, be a system of generators of the monoid formed by the set of non-negative solutions of \eqref{eq:2}. Assume that $\alpha_1,\ldots,\alpha_d$ are the elements in $A$ with the last coordinate equal to zero and $\alpha_{d+1},\ldots,\alpha_g$ are those elements in $A$ with the last coordinate equal to one. Let $\pi:{\mathbb N}^{p+2} \to {\mathbb N}^p$ be the projection onto the first $p$ coordinates. Then the set
	of non-negative integer solutions of \eqref{eq:1} is $\{\pi(\alpha_{d+1}),\ldots,\pi(\alpha_g)\} + \langle \pi(\alpha_1),\ldots,\pi(\alpha_d) \rangle$. Moreover, $A(b-a)=\langle \pi(\alpha_1),\ldots,\pi(\alpha_d) \rangle$.
\end{proposition}

An immediate consequence of the above proposition is the next result.

\begin{corollary}\label{cor12}
	\begin{enumerate}
		\item ${\mathcal D} = \{\pi(\alpha_1),\ldots,\pi(\alpha_d)\}$.
		\item ${\mathcal C}=\mathrm{minimals}_{\leq_{A(b-a)}}\{\pi(\alpha_{d+1}),\ldots,\pi(\alpha_g)\}$.
	\end{enumerate}
\end{corollary}

\begin{remark}\label{rem12}
	Let us observe that it is really easy to determine ${\mathcal C}$ using \eqref{eq:3}. In effect, let ${\mathcal F} = \{(x_1,\ldots,x_{p+2}) \in {\mathbb N}^{p+2} \mid (b_1-a_1)x_1+\cdots+(b_p-a_p)x_p-x_{p+1} - (\alpha + \beta)x_{p+2} = 0 \}$. It is well known (see \cite{monoides}) that ${\mathcal F}$ is a finitely generated monoid of $({\mathbb N}^{p+2},+)$ and, in addition, its set of minimal generators coincide with the set of minimal elements (with respect the usual order in ${\mathbb N}^{p+2}$) of ${\mathcal F} \setminus \{(0,\ldots,0)\}$. Moreover, from \cite{pottier}, we have that, if $\{(x_1,\ldots,x_{p+2})$ is a minimal element of ${\mathcal F} \setminus \{(0,\ldots,0)\}$, then $x_1+\cdots+x_{p+2} \leq |b_1-a_1| +\cdots+ |b_p-a_p| +\alpha+\beta+2$.
\end{remark}

\begin{example}\label{exmp13}
	We are going to compute the minimal system of generators of $S=S\big((4,5),(3,6),3,1\big)\cup\{0\}$ using Algorithm~\ref{alg10}.
	
	Let ${\mathcal F}= \{(x_1,x_2,x_3,x_4)\in {\mathbb N}^4 \mid -x_1+x_2-x_3-4x_4=0 \}$. Having in mind Remark~\ref{rem12}, it is easy to see that
	$$\mathrm{minimals}\big({\mathcal F} \setminus \{(0,0,0,0)\}\big) = \{(1,1,0,0),(0,1,1,0),(0,4,0,1)\}.$$
	Now, by Corollary~\ref{cor12}, we have that ${\mathcal D}=\{(1,1),(0,1)\}$ and ${\mathcal C}=\{(0,4)\}$. Consequently, $D=[9,9]\cup[5,6]=\{5,6,9\}$ and $C=[23,23]=\{23\}$. Finally, by applying Algorithm~\ref{alg8}, the smallest $\{5,6,9\}$-monoid containing $\{23\}$ is 
	$$S=\{0,23,28,29,32,33,34,35,37,\to\}.$$
	By the way, let us observe that $S$ is the set of solutions for the example in the introduction.
\end{example}

\section{A brief remark on numerical semigroups}

By Proposition~\ref{prop2}, we know that $S(a,b,\alpha,\beta) \cup \{0\}$ is a submonoid of $({\mathbb N},+)$. On the other hand, we have that the sets of solutions in Examples~\ref{exmp5}, \ref{exmp9}, and \ref{exmp13} are \emph{numerical semigroups} (that is, a submonoid $S$ of $({\mathbb N},+)$ such that ${\mathbb N} \setminus S$ is finite). In the next example the answer is related with a monoid that is not a numerical semigroup.

\begin{example}
	Let us calculate $S=S\big((4,5),(3,5),0,0\big)$ (observe that $0\in S$). First, we need a system of generators for $A(-1,0)=\{(x,y) \in {\mathbb N}^2 \mid -x+0y \geq 0 \}$. For that, we compute the minimal elements of $B(-1,0) \setminus \{(0,0,0)\}$, where $B(-1,0)=\{ (x,y,z) \in {\mathbb N}^3 \mid -x+0y-z=0 \}$. It is obvious that $(0,1,0)$ is the unique minimal element of $B(-1,0) \setminus \{(0,0,0)\}$ and, in consequence, $\{(0,1)\}$ is a system of generators for $A(-1,0)$. Therefore, by applying Theorem~\ref{thm4}, we conclude that $[5,5] =\{5\}$ is a system of generators for $S$. Thus, $S=\langle 5 \rangle = 5{\mathbb N}$.
\end{example}

At this point, a question arise in a natural way: when is $S(a,b,\alpha,\beta) \cup \{0\}$ a numerical semigroup? In order to give an answer, we will study three cases.

\vspace{6pt} \noindent \textbf{Case~1.} $a_i>b_i$ for all $i\in\{1,\ldots,r\}$ and $(\alpha,\beta) \in {\mathbb N}^2$ or $a_i\geq b_i$ for all $i\in\{1,\ldots,r\}$ and $(\alpha,\beta)\in {\mathbb N}^2\setminus (0,0)$.

It is obvious that $$a_1x_1 +\cdots+ a_px_p + \alpha > b_1x_1 +\cdots+ b_px_p - \beta$$ for all $(x_1,\ldots x_p) \in {\mathbb N}^p$. Thus, $S(a,b,\alpha,\beta) = \emptyset$ and $S(a,b,\alpha,\beta) \cup \{0\} = \{0\}$ (that is, the trivial submonoid).

\vspace{6pt} \noindent \textbf{Case~2.} $a_i\geq b_i$ for all $i\in\{1,\ldots,r\}$,  there exists $i^*\in\{1,\ldots,r\}$ such that $a_{i^*}=b_{i^*}$, and $(\alpha,\beta)=(0,0)$. 

If $E=\big\{ i\in \{1,\ldots,r\} \mid a_i=b_i \big\}$ and $I=\big\{ i\in \{1,\ldots,r\} \mid a_i>b_i \big\}$, then it is clear that $E\cup I=\{1,\ldots,r \}$ and $E\cap I=\emptyset$. On the other hand,
$$a_1x_1 +\cdots+ a_px_p \leq b_1x_1 +\cdots+ b_px_p \Leftrightarrow 0 \leq \sum_{i\in I} (b_i-a_i)x_i.$$
Since $b_i-a_i<0$ for all $i\in I$ and $(x_1,\dots,x_p) \in {\mathbb N}^p$, we have that the last inequality is true if and only if $x_i=0$ for all $i\in I$. Therefore,
$$S(a,b,0,0)=\left\{ n\in {\mathbb N} \Bigm| \sum_{i\in E} a_i x_i \leq n \leq \sum_{i\in E} b_i x_i \mbox{ for some } (x_i)_{i\in E} \in {\mathbb N}^r \right\},$$
where $r$ is the cardinality of $E$. Now, since $\sum_{i\in E} a_i x_i = \sum_{i\in E} b_i x_i$ for all $(x_i)_{i\in E} \in {\mathbb N}^r$, we conclude that $S(a,b,0,0)$ is the submonoid generated by the set $A=\{a_i \mid i\in E\}$. In addition, $S(a,b,0,0)$ will be a numerical semigruoup if and only if $\gcd(A)=1$ (see \cite[Lemma~2.1]{springer}).

\vspace{6pt} \noindent \textbf{Case~3.} There exists $j\in\{1,\ldots,r\}$ such that $a_{j}<b_{j}$. 

It is clear that $(b_j-a_j)x_j-\beta-\alpha>0$ for a suitable choice of $x_j$. Therefore, the numerical semigroup generated by $\{a_jx_j+\alpha,a_jx_j+\alpha+1 \}$ is a subset of $S(a,b,\alpha,\beta) \cup \{0\}$ which, in consequence, is also a numerical semigroup.

\section{Conclusion}

Starting from a real world situation, the aim of this work has been to compute the set
$$S(a,b,\alpha,\beta)=\left\{ \begin{array}{c|c} 
\!\!\! n\in {\mathbb N} \! & \begin{array}{c}
\!\!\! a_1x_1 +\cdots+ a_px_p + \alpha \leq n \leq b_1x_1 +\cdots+ b_px_p - \beta \\[3pt] \mbox{ for some } (x_1,\ldots x_p) \in {\mathbb N}^p
\end{array}
\end{array} \!\!\!\!\! \right\},$$
where $a=(a_1,\ldots,a_p), b=(b_1,\ldots,b_p) \in {\mathbb N}^p$ and $\alpha,\beta \in {\mathbb N}$.

We have achieved our purpose in two steps. Firstly, we have studied the case $(\alpha,\beta)=(0,0)$. Secondly, by means of an algorithm, we have solved the case $(\alpha,\beta)\not=(0,0)$.

It is interesting to observe that, in general, $S(a,b,\alpha,\beta) \cup \{0\}$ is a submonoid of $({\mathbb N},+)$ and, in some cases, a numerical semigroup.


\end{document}